\documentclass[a4paper,reqno]{amsart}

\makeatletter
\@namedef{subjclassname@2020}{%
  \textup{2020} Mathematics Subject Classification}
\makeatother

\usepackage{amsmath}
\usepackage{amssymb}
\usepackage{amsthm}
\usepackage{mathrsfs}
\usepackage{latexsym}
\usepackage{amscd}
\usepackage{xypic}
\xyoption{curve}
\usepackage{ifthen} 
\usepackage{hyperref} 
\usepackage{graphicx}
\usepackage{enumerate}
\usepackage{enumitem}
\usepackage{rotating}
\usepackage{xcolor}
\usepackage{mathtools}
\usepackage{todonotes}
\usepackage{color,soul}

\usepackage{float}
\restylefloat{table}

\numberwithin{equation}{section}

\def\cocoa{{\hbox{\rm C\kern-.13em o\kern-.07em C\kern-.13em o\kern-.15em A}}}

\newtheorem{theorem}{Theorem}[section]

\newtheorem{lemma}[theorem]{Lemma}
\newtheorem{proposition}[theorem]{Proposition}
\newtheorem{corollary}[theorem]{Corollary}

\theoremstyle{definition}
\newtheorem{remark}[theorem]{Remark}
\newtheorem{definition}[theorem]{Definition}
\newtheorem{example}[theorem]{Example}

\newcommand {\sHom}{\mathcal{H}\kern -0.25ex{\mathit om}}
\newcommand {\sExt}{\mathcal{E}\kern -0.25ex{\mathit xt}}
\newcommand {\sTor}{\mathcal{T}\kern -0.25ex{\mathit or}}

\newcommand {\rk}{\mathrm{rk}}

\newcommand {\Hilb}{\mathcal{H}\kern -0.25ex{\mathit ilb\/}}

\newcommand {\cA}{\mathcal{A}}
\newcommand {\cB}{\mathcal{B}}

\newcommand{\cE}{{\mathcal E}}
\newcommand{\cF}{{\mathcal F}}

\newcommand{\cN}{{\mathcal N}}
\newcommand{\cO}{{\mathcal O}}

\newcommand{\cI}{{\mathcal I}}

\newcommand {\bZ}{\mathbb{Z}}

\newcommand {\bC}{\mathbb{C}}
\newcommand {\bP}{\mathbb{P}}

\newcommand{\Pic}{\operatorname{Pic}}

\def\p#1{{\bP^{#1}}}

\title[A remark by Daniel Ferrand on bundles on Fano threefolds]{A remark by Daniel Ferrand on bundles on Fano threefolds}

\author[G. Casnati, A.F. Lopez]{Gianfranco Casnati* and Angelo Felice Lopez**}

\address{\hskip -.43cm Gianfranco Casnati, Dipartimento di Scienze Matematiche, Politecnico di Torino, c.so Duca degli Abruzzi 24,
10129 Torino, Italy. email: {\tt gianfranco.casnati@polito.it}}

\address{\hskip -.43cm Angelo Felice Lopez, Dipartimento di Matematica e Fisica, Universit\`a di Roma
Tre, Largo San Leonardo Murialdo 1, 00146, Roma, Italy. e-mail {\tt angelo.lopez@uniroma3.it}}

\thanks{* The  author is a member of GNSAGA group of INdAM and partially supported by the framework of the MIUR grant Dipartimenti di Eccellenza 2018-2022 (E11G18000350001).}

\thanks{** Research partially supported by PRIN ``Advances in Moduli Theory and Birational Classification'', GNSAGA--INdAM and the MIUR grant Dipartimenti di Eccellenza 2018-2022.}


\begin{document}

\begin{abstract}
Let $X$ be a Fano threefold with index $i_X$ and fundamental line bundle $\cO_X(h)$. We study $\mu$--semistable rank two bundles $\cE$ on $X$ with $c_1(\cE)=0$, $h^0(\cE) \ne 0$ and $h^1(\cE(-\lceil\frac{i_X}{2}\rceil h))=0$.
\end{abstract}

\maketitle

\section{Introduction}

An instanton bundle on the projective space $\p3$ over the complex field $\bC$ is a rank two bundle such that $h^0(\cE)=h^1(\cE(-2))=0$ and $c_1(\cE)=0$.

Instanton bundles on $\p3$ were first introduced in the seminal paper \cite{A--D--H--M} and widely studied from different viewpoints since the discovery of their connection, through the Atiyah--Penrose--Ward transformation, with the solutions of the Yang--Mills equations (see \cite{A--W}). 

The projective space $\p3$ is the very first example of {\sl Fano threefold}, that is of a smooth threefold $X$ whose anticanonical line bundle $\omega_X^{-1}$ is ample: see \cite{I--P} for a survey on the classification and properties of such threefolds. In particular, the {\sl index of $X$} is the greatest $i_X\in\bZ$ such that $\omega_X\cong\cO_X(-i_Xh)$ for an ample $\cO_X(h) \in \Pic(X)$. It is well known that $1\le i_X\le 4$ and that such an $\cO_X(h)$ is uniquely determined: it is called the {\sl fundamental line bundle} of $X$. Notice that $i_X=4$ if and only if $X\cong\p3$ and $i_X=3$ if and only if $X\subset \p4$ is a smooth quadric.

Because of its importance, the notion of instanton bundle has been extended to Fano threefolds. In \cite{Fa2, Kuz} the authors defined for the first time instanton bundles on Fano threefolds $X$ with Picard number $\varrho_X=1$. In a second stage such a restriction has been removed (e.g. see \cite{M--M--PL, C--C--G--M}). In particular, in \cite{A--C--G2} (see also \cite{An--Cs}) the following general definition has been suggested. In order to fix the notation, if $\varepsilon \in \{0, 1\}$ we set 
$$
q_X^\varepsilon=\Bigl\lceil\frac{i_X-\varepsilon}{2}\Bigr\rceil.
$$
An {\sl instanton bundle} $\cE$ on $X$ is a rank two bundle $\cE$ such that the following properties hold:
\begin{itemize}
\item $c_1(\cE)=-\varepsilon h$, where $\varepsilon \in \{0, 1\}$;
\item $\cE$ is $\mu$--semistable with respect to $\cO_X(h)$ and $(1-\varepsilon)h^0(\cE)=0$;
\item $h^1(\cE(-q_X^\varepsilon h))=0$.
\end{itemize}
The instanton bundle  $\cE$ with $c_1(\cE)=-\varepsilon h$ is called even or odd according to the parity of $\varepsilon$.

On the one hand, the vanishing $(1-\varepsilon)h^0(\cE)=0$ is equivalent to the $\mu$--stability of $\cE$ when $X$ has Picard number $\varrho_X=1$ and $\cE$ is even. On the other hand, if $\varrho_X=1$ and $\cE$ is odd, then the notions of $\mu$--stability and $\mu$--semistability actually coincide (see for example Lemma \ref{lHoppe}). This is no longer true if $\varrho_X\ge2$, because many pathologies can occur in this case, as showed in \cite{M--M--PL, An--Ma1}.

The above discussion motivates our interest in studying rank two bundles $\cE$ on a Fano threefold $X$ such that $c_1(\cE)=0$, $h^1(\cE(-q_X^0 h))=0$, $h^0(\cE)\ne 0$, and which are $\mu$--semistable with respect to $\cO_X(h)$. 

Certainly $\cO_X^{\oplus2}$ is strictly $\mu$--semistable, it satisfies the above restrictions and it always occurs on each Fano threefold. Hence our interest is focused on the case $\cE\not\cong\cO_X^{\oplus2}$: in this case  we are able to prove the result below in Section \ref{sFerrand}.

\begin{theorem}
\label{tFerrand}
Let $X$ be a smooth Fano threefold with fundamental line bundle $\cO_X(h)$. 

Let $\cE \not \cong\cO_X^{\oplus 2}$ be a $\mu$--semistable rank two bundle on $X$ with $c_1(\cE)=0$, $h^0(\cE) \ne 0$ and $h^1(\cE(-q_X^0h))=0$.

Then $i_X \le 3$, $\cE$ is indecomposable, $h^0(\cE)=1$ and each non--zero section of $\cE$ vanishes exactly on the same locally complete intersection curve $Z\subset X$ such that $\omega_Z\cong\cO_Z\otimes\cO_X(-i_Xh)$. 
Moreover, the following assertions hold:
\begin{enumerate}
\item If $i_X=3$, then the reduced structure of $Z$ is a union of disjoint $h$--lines, $Z$ is everywhere nonreduced and each connected component of $Z$ has even $h$--degree.
\item If $i_X=2$, then the reduced structure of $Z$ is a union of disjoint $h$--lines. Moreover, let $\overline Z$ be any connected component of $Z$ and let $L=\overline Z_{red}$ be the $h$--line. Then, either $\overline Z =L$ is reduced or the extension $L \subset \overline Z$ is quasi-primitive and $\overline Z \cong \bP^1 \times {\rm Spec}(\mathbb C[x]/x^{k+1})$ or the extension is thick and the multiple structure is defined by successive extensions $L=Z_0 \subset Z_1 \subset \ldots \subset Z_k=\overline Z$ with 
$$0 \to \cO_L^{\oplus r_j} \to \cO_{Z_j} \to \cO_{Z_{j-1}} \to 0$$
for $1 \le j \le k$ and $r_1=2$.
\item If $i_X=1$, then each connected component of $Z$ has even $h$--degree and its reduced structure is a  nodal seminormal tree of smooth rational curves with each vertex of degree at most $3$.
\end{enumerate}
\end{theorem}
We note that, in order to get a classification result for bundles as in the above theorem, a description of the possible multiple structures would be needed. Aside from the case $i_X=2$, some more cases are studied in Section \ref{sQuadric}. 

Throughout the whole paper we work over the complex numbers. By {\sl curve (inside a scheme $X$)} we mean any closed subscheme of pure dimension $1$ contained in $X$. Moreover, for the notions of $h$--degree, $h$--line and seminormal tree of smooth rational curves we refer the reader to Definitions \ref{tr} and \ref{tr2} respectively.

When $\varrho_X=1$, an immediate corollary of the above theorem gives a complete characterization of strictly $\mu$--semistable rank two bundles $\cE$ on $X$ with $c_1(\cE)=0$ and $h^1(\cE(-q_X^0h))=0$, because in this case it is automatically true that $h^0(\cE)\ne0$ (see Lemma \ref{lHoppe}). 

In particular, the above result is well known when $X\cong\p3$ (and this motivated the title of the paper): it is a remark by Daniel Ferrand, see \cite[Remark at page 136]{O--S--S} where a draft of its proof is given in this particular case. Moreover, it is also well known that the similar problem of classifying semistable (in the sense of Gieseker) but not $\mu$--stable rank two bundles $\cE$ on $X$ with $c_1(\cE)=0$ leads immediately to the isomorphism $\cE\cong\cO_X^{\oplus2}$: see Remark \ref{rGieseker}.

In Section \ref{sExample} we first collect several examples of rank two bundles showing that all the cases above actually occur when $i_X \ge 2$. We close that section with Example \ref{prime} where we deal with the case $i_X=1$ and $\varrho_X=1$, which needs some more care. Indeed, in this case, the picture is slightly less precise than the case $i_X \ge 2$, because, as we show, it is not possible to find sharper bounds on the degree of the reduced structure of the connected components of $Z$. 

Finally, in Section \ref{sQuadric}, we carefully deal with strictly $\mu$--semistable  rank two bundles $\cE$ on the smooth quadric hypersurface $X\subset\p4$ with $c_1(\cE)=0$ and $c_2(\cE)h=2$. In particular we show that, in this case, the double structure is described by Ferrand's construction, see Remark \ref{Ferrand}. 

\section{General results}

In this section we recall some definitions and general facts which will be used later. 

For the reader's convenience we recall the Riemann--Roch formula for a locally free sheaf $\cA$ on a smooth threefold $X$:

\begin{equation}
\label{RRgeneral}
\begin{aligned}
\chi(\cA)&=\rk(\cA)\chi(\cO_X)+{\frac16}(c_1(\cA)^3-3c_1(\cA)c_2(\cA)+3c_3(\cA)) \\
&-{\frac14}(\omega_Xc_1(\cA)^2-2\omega_Xc_2(\cA))+{\frac1{12}}(\omega_X^2c_1(\cA)+c_2(\Omega^1_{X}) c_1(\cA)).
\end{aligned}
\end{equation}

In what follows we recall some helpful definitions and facts concerning the stability notion of vector bundles.

\begin{definition}
\label{dStable}
Let $X$ be a smooth irreducible variety of dimension $n\ge1$ endowed with an ample line bundle $\cO_X(H)$.

If $\cA$ is any torsion--free sheaf we define the {\sl slope of $\cA$ (with respect to $\cO_X(H)$)} as
$$
\mu(\cA):=\frac{c_1(\cA)H^{n-1}}{\rk(\cA)}.
$$
The torsion--free sheaf $\cA$ is {\sl $\mu$--semistable} (resp. {\sl $\mu$--stable}) if for all proper subsheaves $\cB$ with $0<\rk(\cB)<\rk(\cA)$ we have $\mu(\cB) \le \mu(\cA)$ (resp. $\mu(\cB)<\mu(\cA)$). 
\end{definition}

The problem of determining when a vector bundle $\cE$ on a variety $X$ is $\mu$--stable, or even only $\mu$--semistable, is in general hard without any further assumptions on $X$ and $\cE$. On the other hand, when $\varrho_X=1$, we have the following well known result.

\begin{lemma}
\label{lHoppe}
Let $X$ be a smooth irreducible variety of dimension $n\ge1$ and let $\cO_X(H)$ be an ample line bundle which generates $\Pic(X)$.

If $\cE$ is a rank two bundle on $X$, then the following assertion hold:
\begin{enumerate}
\item if $c_1(\cE) \in \{0, -H\}$, then $\cE$ is $\mu$--stable if and only if $h^0(\cE)=0$;
\item if $c_1(\cE)=0$, then $\cE$  is $\mu$--semistable  if and only if $h^0(\cE(-H))=0$; 
\item if $c_1(\cE)=-H$, then $\cE$ is $\mu$--stable if and only if it is $\mu$--semistable.
\end{enumerate}
\end{lemma}
\begin{proof}
See \cite[Lemmas II.1.2.3 and II.1.2.5]{O--S--S}. Though  the results are stated therein for $X=\p n$, it is easy to check that the proofs hold for every $X$ and $H$ as in the statement, when $H$ is effective. Passing to $mH, m \gg 0$, one gets the general case of $H$ ample. The lemma also follows easily by \cite[Corollary 4]{J--M--P--SE}.
\end{proof}

We close this section recalling what we need about the relation between subschemes of codimension $2$ and rank two bundles. To this purpose we observe that for each subscheme $Z\subset X$ there exists the exact sequence 
\begin{equation}
\label{seqStandard}
0 \longrightarrow \cI_{Z/X} \longrightarrow \cO_X \longrightarrow \cO_Z \longrightarrow 0.
\end{equation}
Let $\cF$ be a rank two bundle on a smooth irreducible variety $X$ of dimension $n \ge 2$ and let $0 \ne s \in H^0(\cF)$. In general, its {\sl zero--locus} $(s)_0\subset X$ is either empty or its codimension is at most
$2$. We can always write 
\begin{equation}
\label{sez}
(s)_0=Y\cup Z
\end{equation}
where $Z$ has pure codimension $2$ or it is empty, and $Y$ has pure codimension $1$ or it is empty. In particular $\cF(-Y)$ has a non--zero section vanishing on $Z$, thus we can consider the following exact sequence induced by its Koszul complex 
\begin{equation}
\label{seqSerre}
0 \longrightarrow \cO_X(Y) \longrightarrow \cF \longrightarrow \cI_{Z/X}(-Y+c_1(\cF)) \longrightarrow 0.
\end{equation}
Tensoring \eqref{seqSerre} by $\cO_Z$ yields $\cI_{Z/X}/\cI^2_{Z/X}\cong\cF^\vee(Y)\otimes\cO_Z$, whence the normal bundle of $Z$ inside $X$ satisfies $\cN_{Z/X}\cong\cF(-Y)\otimes\cO_Z$. If $Y=0$, then $Z$ is locally complete intersection inside $X$, because $\rk(\cF)=2$. In particular, it has no embedded components and the adjunction formula
\begin{equation}
\label{omega1}
\omega_Z\cong\cO_Z\otimes\omega_X(c_1(\cF))
\end{equation}
holds for such a scheme as well (see \cite[Proof of Proposition III.7.2, page 180]{H}).

The Serre correspondence allows us to revert the above construction as follows.

\begin{theorem}
\label{tSerre}
Let $X$ be a smooth irreducible variety $X$ of dimension $n \ge 2$ and let $Z\subset X$ be a locally complete intersection subscheme of pure codimension $2$.
  
If $\det(\cN_{Z/X})\cong\cO_Z\otimes\mathcal L$ for some $\mathcal L\in\Pic(X)$ such that $H^2(\mathcal L^\vee)=0$, then there exists a rank two bundle $\cF$ on $X$ such that:
\begin{enumerate}
\item $\det(\cF)\cong\mathcal L$;
\item $\cF$ has a global section $s$ such that $Z$ coincides with the zero locus $(s)_0$ of $s$.
\end{enumerate}
Moreover, if $H^1({\mathcal L}^\vee)= 0$, the above two conditions  determine $\cF$ up to isomorphism.
\end{theorem} 
\begin{proof}
See \cite[Theorem 1.1]{Ar}.
\end{proof}

For further notation and all the other necessary results not explicitly mentioned in the paper, unless otherwise stated we tacitly refer to \cite{Ha2}.

\section{The proof of Theorem \ref{tFerrand}}

\label{sFerrand}

This section contains some preliminary results needed in the next sections. We begin with the following definitions.
\begin{definition}
\label{tr}
Let $X$ be a smooth variety and let $\cO_X(H)$ be an ample line bundle on $X$. 

The {\sl $H$--degree} of a curve $Z \subset X$ is $HZ$. An {\sl $H$--line} is a smooth irreducible rational curve $L \subset X$ such that $HL=1$. 
\end{definition}

\begin{definition}
\label{tr2}
Let $C$ be a reduced connected curve. 

We say that $C$ is a {\sl seminormal tree of smooth rational curves} if its irreducible components are smooth rational curves, the singularities are given by smooth branches with independent tangents (that is they are seminormal singularities) and the dual graph is a tree.
\end{definition}

Seminormal trees of smooth rational curves are characterized as the reduced connected curves with arithmetic genus $0$ in \cite[Proposition 1.8]{C}. Moreover, still by \cite[Proposition 1.8]{C}, they are nodal if and only if they are Gorenstein. Since they have independent tangents, when there are more than two curves meeting at a point, they are not nodal. On the other hand, when $C$ lies on a smooth threefold, there can be at most three curves meeting at a point. In the latter case, that is the one we consider in this paper, the dual graph is a tree with each vertex of degree at most $3$. 

\begin{lemma}
\label{lFerrand}
Let $X$ be a smooth Fano threefold with fundamental line bundle $\cO_X(h)$. 

Let $\cE$ be a rank two bundle on $X$ such that $c_1(\cE)=\varepsilon h$ and $h^1(\cE(-eh))=0$ and $\cE  \not \cong \cO_X \oplus \cO_X(\varepsilon h)$, for some $\varepsilon, e \in \bZ$ with $1 \le 1+\varepsilon \le e \le  i_X$. Assume that there is $s \in H^0(\cE)$ such that $Z=(s)_0 \subset X$ is a curve.

Then one of the following assertions hold:
\begin{enumerate}
\item $e=i_X-1$ and $Z_{red}$ is a union of disjoint $h$--lines; 
\item $e=i_X$ and each connected component of $Z_{red}$ is a seminormal tree of smooth rational curves with each vertex of degree at most $3$.
\end{enumerate}
\end{lemma}
\begin{proof}
Since $Z$ is a curve, it follows that the sequence \eqref{seqSerre} corresponding to $Z$ becomes
\begin{equation}
\label{seqFerrand}
0 \longrightarrow\cO_{X} \longrightarrow \cE \longrightarrow \cI_{Z/X}(\varepsilon h) \longrightarrow 0,
\end{equation}
where either $Z=\emptyset$ or $Z\subset X$ is a locally complete intersection 
curve and \eqref{omega1} gives
\begin{equation}
\label{omega}
\omega_Z\cong\cO_Z\otimes\cO_X((\varepsilon-i_X) h).
\end{equation}

If $Z=\emptyset$, then \eqref{seqFerrand} splits,  because $h^1(\cO_X(-\varepsilon h))=0$ for each Fano threefold  by Kodaira vanishing, since $\varepsilon \ge 0$. Hence $\cE\cong\cO_X\oplus\cO_X(\varepsilon h)$, contradicting our hypothesis.

It remains to handle the case $Z\ne\emptyset$: we will deal with it in what follows. Since $e \ge 1$ and $\varepsilon-e\le -1$, we have that $h^i(\cO_X(-eh))=h^i(\cO_X((\varepsilon-e)h))=0$ for $i\le 2$. Hence the cohomologies of \eqref{seqStandard} tensored by $\cO_X((\varepsilon-e) h)$ and \eqref{seqFerrand} tensored by $\cO_X(-e h)$ imply
\begin{equation}
\label{instanton}
h^0(\cO_Z\otimes\cO_X((\varepsilon-e)h))=h^1(\cI_{Z/X}((\varepsilon-e)h))=h^1(\cE(-eh))=0.
\end{equation}
Each connected component of $Z$ still satisfies the same properties of the whole $Z$, namely it is a locally complete intersection curve and the analogue of \eqref{omega} holds for it. We restrict to a fixed connected component $\overline{Z}$ of $Z$ from now on. 

Now \eqref{omega} and \eqref{instanton} give
$$h^1(\cO_Z\otimes\cO_X((e-i_X)h)=h^0(\omega_Z\otimes\cO_X((-e+i_X)h)=h^0(\cO_Z\otimes\cO_X((\varepsilon-e)h))=0.$$
Then the exact sequence
$$
0 \longrightarrow \cI_{\overline{Z}/Z}\otimes\cO_X((e-i_X)h) \longrightarrow \cO_Z\otimes\cO_X((e-i_X)h)\longrightarrow \cO_{\overline{Z}}\otimes\cO_X((e-i_X)h)\longrightarrow 0
$$
and the fact that $H^2(\cI_{\overline{Z}/Z}\otimes\cO_X((e-i_X)h))=0$ since $Z$ is a curve, imply that
\begin{equation}
\label{h^1}
h^1(\cO_{\overline{Z}}\otimes\cO_X((e-i_X)h))=0.
\end{equation}
If $\overline{Z}_{red}$ is the reduced scheme structure on $\overline{Z}$, then
\begin{equation}
\label{Zprime}
h^0(\cO_{\overline{Z}_{red}})=1,\qquad p_a(\overline{Z}_{red}) \ge 0.
\end{equation}
Let $\overline{Z}_{red}^{(1)},\dots,\overline{Z}_{red}^{(r)}$ be the irreducible components of $\overline{Z}_{red}$. Then, for any $1 \le i \le r$ and $t \ge 1$, we have that $\cO_{\overline{Z}_{red}^{(i)}}\otimes\cO_X(th)$ is ample, hence $H^0(\cO_{\overline{Z}_{red}^{(i)}}\otimes\cO_X(-th))=0$ (see for example \cite[Proposition 1]{M}). Hence the inclusion $\cO_{\overline{Z}_{red}}\subseteq\bigoplus_{i=1}^r\cO_{\overline{Z}_{red}^{(i)}}$ yields
\begin{equation}
\label{h^0}
H^0(\cO_{\overline{Z}_{red}}\otimes\cO_X(th))=0 \ \hbox{for} \ t\le -1. 
\end{equation}
Moreover the exact sequence
$$
0\longrightarrow \cI_{\overline{Z}_{red}/\overline{Z}}\otimes\cO_X((e-i_X)h) \longrightarrow \cO_{\overline{Z}}\otimes\cO_X((e-i_X)h)\longrightarrow \cO_{\overline{Z}_{red}}\otimes\cO_X((e-i_X)h)\longrightarrow 0
$$
and \eqref{h^1} show that $h^1(\cO_{\overline{Z}_{red}}\otimes\cO_X((e-i_X)h))=0$. Hence, applying the Riemann--Roch theorem \cite[Theorem 7.3.26]{L} on $\overline{Z}_{red}$, we get
\begin{equation}
\label{Genus}
\begin{aligned}
h^0(\cO_{\overline{Z}_{red}}\otimes\cO_X((e-i_X)h))&=\chi(\cO_{\overline{Z}_{red}}\otimes\cO_X((e-i_X)h))=\\
&=(e-i_X)\overline{Z}_{red}h+1-p_a(\overline{Z}_{red}).
\end{aligned}
\end{equation}

If $e=i_X$, then \eqref{Zprime} yields $p_a(\overline{Z}_{red})=0$,  hence $\overline{Z}_{red}$ is a seminormal 
tree of smooth rational curves with each vertex of degree at most $3$ by \cite[Proposition 1.8]{C}. This gives (2). 

If $e \le i_X-1$, then \eqref{Zprime}, \eqref{h^0} and \eqref{Genus} imply
$$
0\le p_a(\overline{Z}_{red})=(e-i_X)\overline{Z}_{red}h+1\le0,
$$
hence $p_a(\overline{Z}_{red})=0, i_X=e+1$ and $\overline{Z}_{red}$ is an $h$--line. 
This gives (1).
\end{proof}

\begin{remark}
The cases listed in Lemma \ref{lFerrand} are not mutually exclusive, because we have not assumed anywhere that $e$ is the minimal integer such that $h^1(\cE(-eh))=0$. 
\end{remark}

We are now ready to prove Theorem \ref{tFerrand} stated in the Introduction. 
\medbreak
\noindent{\it Proof of Theorem \ref{tFerrand}.}

Recall that we are assuming $\cE\not\cong\cO_X^{\oplus2}$. We first show that $\cE$ is indecomposable. In fact, assume to the contrary that $\cE$ is decomposable. Since $c_1(\cE)=0$ and $h^0(\cE) \ne 0$, there is a line bundle $\mathcal L$ on $X$ such that $\cE \cong \mathcal L \oplus \mathcal L^{-1}$ and $h^0(\mathcal L) \ne 0$. Now, if $D \in \vert \mathcal L \vert$, we get by the $\mu$--semistability of $\cE$ that $Dh^2 = \mu(\cO_X(D)) \le \mu(\cE) = 0$. But $D$ is effective and $h$ is ample, hence $D$ must be trivial and therefore $\cE \cong \cO_X^{\oplus2}$, a contradiction. This proves that $\cE$ is indecomposable.

Let $s\ne0$ be a section in $H^0(\cE)$ and write $(s)_0=Y\cup Z$ as in \eqref{sez}. If $Y \ne \emptyset$, then \eqref{seqSerre} gives an inclusion $\cO_X(Y)\subset \cE$ and $\mu(\cO_X(Y))=Yh^2>0=\mu(\cE)$, which contradicts the $\mu$--semistability of $\cE$.

We deduce that $Z=(s)_0$ and either $Z = \emptyset$ or $Z \subset X$ is a locally complete intersection curve. Hence the sequence \eqref{seqSerre} associated to $s$ becomes
\begin{equation}
\label{seqBundle}
0\longrightarrow\cO_X\longrightarrow\cE\longrightarrow\cI_{Z/X}\longrightarrow0.
\end{equation}
If $Z = \emptyset$, we have that $\cI_{Z/X} \cong \cO_X$ and the sequence \eqref{seqBundle} splits, since $h^1(\cO_X)=0$. But then $\cE \cong\cO_X^{\oplus2}$, a contradiction. Thus $Z$ is a locally complete intersection curve.

The cohomology of \eqref{seqBundle} yields $h^0(\cE)=1$, hence the zero locus of each non--zero section in $H^0(\cE)$ is $Z$. Moreover \eqref{omega1} gives
\begin{equation}
\label{BB}
\omega_Z\cong\cO_Z\otimes\cO_X(-i_Xh).
\end{equation}
Let $i_X=1$. In this case $q_X^0=i_X$, hence the connected components of the reduced scheme structure on $Z$ are seminormal trees of smooth rational curves with each vertex of degree at most $3$ by assertion (2) of Lemma \ref{lFerrand}.

Let $i_X \ge 2$. If $i_X=4$ we have that $q^0_X=2$, hence Lemma \ref{lFerrand} with $(\varepsilon, e)=(0, 2)$ implies that $e \in \{3, 4\}$, a contradiction. Therefore $2 \le i_X \le 3$ and we have $q^0_X=i_X-1$, hence again Lemma \ref{lFerrand} implies that the reduced structure of $Z$ is a union of disjoint $h$--lines.

Moreover, if $i_X=3$, then \eqref{BB} induces a similar isomorphism for the dualizing sheaf of each connected component of $Z$. Thus each such component $\overline{Z}$ must be nonreduced, for otherwise, being a disjoint union of $h$--lines, we would have that $\omega_{\overline{Z}}\cong\cO_{\overline{Z}}\otimes\cO_X(-2h)$, a contradiction.

As we already pointed out in the proof of Lemma \ref{lFerrand}, each connected component $\overline{Z}\subset Z$ satisfies the same properties of $Z$. In particular, it is a local complete intersection inside $X$. Moreover, its normal bundle is the restriction of the normal bundle of $Z$, whence $\det(\cN_{\overline{Z}/X}) \cong\cO_{\overline{Z}}$.

Thus Theorem \ref{tSerre} yields the existence of a rank two bundle $\overline{\cE}$ with $c_1(\overline{\cE})=0$ fitting into an exact sequence of the form 
$$
0\longrightarrow\cO_{X}\longrightarrow\overline{\cE}\longrightarrow \cI_{\overline{Z}/X} \longrightarrow 0. 
$$
Then \eqref{RRgeneral} yields 
$$
\chi(\overline{\cE})=2-\frac{i_X}2c_2(\overline{\cE})h,
$$
hence, if $i_X$ is odd, $\overline{Z}h=c_2(\overline{\cE})h$ is necessarily even. 

To finish the proof of the theorem, assume that $i_X=2$ and let $\overline Z$ be a connected component of $Z$. As in the proof of Lemma \ref{lFerrand}, we have that 
$$\omega_Z\cong\cO_Z\otimes\cO_X(-2h) \ \hbox{and} \ H^1(\cO_{\overline{Z}}\otimes\cO_X(-h))=0$$
hence
\begin{equation}
\label{chi}
H^0(\cO_{\overline{Z}}\otimes\cO_X(-h))=H^1(\cO_{\overline{Z}}\otimes\cO_X(-h))=0.
\end{equation}
Consider the $h$--line $L=\overline Z_{red}$ and let $L^{(1)}$ be the first infinitesimal neighborhood of $L$ in $X$. According to \cite[\S 3]{bf}, we have the corresponding CM filtration
$$L=Z_0 \subset Z_1 \subset \ldots \subset Z_k=\overline Z$$
where, for each $1 \le j \le k$, $\cE_j := {\cI_{Z_{j-1}/Z_j}}$ is a locally free sheaf of rank $r_j$ supported on $L$ and either:
\begin{itemize}
\item [(a)] ${\rm rank} (\cE_1)=0$ and $\overline Z=L$ is reduced, or
\item [(b)]  ${\rm rank} (\cE_1)=1$ (equivalently, $L^{(1)} \not\subset \overline Z$) and the extension is called quasi-primitive, or 
\item [(c)] ${\rm rank} (\cE_1)=2$ (equivalently, $L^{(1)} \subset \overline Z$) and the extension is called thick.
\end{itemize}
In case (b), we have $\cE_1 = \cO_{\p1}(\alpha)$ and we claim that $\alpha \ge 0$. Indeed, \eqref{chi} gives that $H^1(\cO_{Z_k}\otimes\cO_X(-h))=0$ and the exact sequences
$$0 \to \cE_j(-h) \to \cO_{Z_j}(-h) \to \cO_{Z_{j-1}}(-h) \to 0$$
show inductively that $H^1(\cO_{Z_1}\otimes\cO_X(-h))=0$. Then, the exact sequence
$$0 \to \cO_{\p1}(\alpha-1) \to \cO_{Z_1}(-h) \to \cO_L(-h) \to 0$$
implies that $H^1(\cO_{\p1}(\alpha-1))=0$, hence $\alpha \ge 0$. Since $\chi(\cO_{\overline{Z}}\otimes\cO_X(-h))=0$ by \eqref{chi}, it follows, precisely as in the proof of \cite[Theorem 1.4]{befr}, that the extension is primitive. Then, exactly as in the proof of \cite[Proposition 4.6]{ammpl}, we deduce that $\overline Z \cong \bP^1 \times {\rm Spec}(\mathbb C[x]/x^{k+1})$.

In case (c), precisely as in the proof of \cite[\S 4.2]{ammpl} but using now \eqref{chi}, we get that $\cE_j \cong \cO_L^{\oplus r_j}$ and therefore the multiple structure is defined by extensions
$$0 \to \cO_L^{\oplus r_j} \to \cO_{Z_j} \to \cO_{Z_{j-1}} \to 0.$$ 
This completes the proof of the theorem.
\qed
\medbreak

\begin{remark}
If $i_X=1$ and $\overline{Z}=\overline{Z}_{red}$, then $\overline{Z}h=2$. Indeed $\omega_{\overline{Z}}\cong\cO_{\overline{Z}}\otimes\cO_X(-h)$ by \eqref{BB}, hence
$$
-\overline{Z}h+1=\chi(\omega_{\overline{Z}})=-h^1(\omega_{\overline{Z}})=-h^0(\cO_{\overline{Z}})=-1,
$$
thanks to the Riemann--Roch theorem on $\overline{Z}$, the connectedness of $\overline{Z}$ and $p_a(\overline{Z})=0$.
\end{remark}

\begin{remark}
\label{quadrica}
Consider the case $i_X=3$ of Theorem \ref{tFerrand}, that is when $X$ is a smooth quadric threefold. We still have, as in the case $i_X=2$, that if $\overline Z$ is a connected component of $Z$, then $L=\overline Z_{red}$ is an $h$--line. We claim that the extension $L \subset \overline Z$ is not thick. In fact, notice first that $\cI_{{L^{(1)}/X}} = \cI_{{L/X}}^2$ and $\cN_{L/X} \cong \cO_{\p1} \oplus \cO_{\p1}(1)$, hence the exact sequence
$$0 \to \cI_{{L/X}}^2 \otimes\cO_X(-h) \to \cI_{{L/X}} \otimes\cO_X(-h) \to \cO_{\p1}(-1) \oplus \cO_{\p1}(-2) \to 0$$
gives $h^2(\cI_{{L/X}}^2 \otimes\cO_X(-h))=h^1(\cO_{\p1}(-1) \oplus \cO_{\p1}(-2))=1$. Now, assume that $L^{(1)} \subset \overline Z$. As in the proof of Lemma \ref{lFerrand}, we have that $H^1(\cO_{\overline{Z}}\otimes\cO_X(-h))=0$ and the exact sequence
$$0 \to \cI_{{L^{(1)}/\overline{Z}}}\otimes\cO_X(-h) \to \cO_{\overline{Z}}\otimes\cO_X(-h) \to \cO_{L^{(1)}}\otimes\cO_X(-h) \to 0$$
shows that $H^1(\cO_{L^{(1)}}\otimes\cO_X(-h))=0$. Then the exact sequence
$$0 \to \cI_{{L/X}}^2 \otimes\cO_X(-h) \to \cO_X(-h) \to \cO_{L^{(1)}}\otimes\cO_X(-h) \to 0$$
implies the contradiction $H^2(\cI_{{L/X}}^2 \otimes\cO_X(-h))=0$.
\end{remark}

\begin{remark}
\label{rGieseker}
The same argument used in \cite[p.~89]{O--S--S} proves that if $X$ is a smooth variety endowed with an ample line bundle $\cO_X(h)$ which generates $\Pic(X)$, $n\ge2$ and  $h^1(\cO_X)=0$, then every rank $2$ bundle $\cE$ which is semistable with respect to $\cO_X(h)$ is either $\mu$--stable or isomorphic to $\cO_X^{\oplus2}$.
\end{remark}

\begin{remark}
Recall that a {\sl Del Pezzo variety} is a pair $(X,\cO_X(H))$ where $X$ is a smooth variety of dimension $n\ge1$ and $\cO_X(H)$ is an ample line bundle such that $\omega_X^{-1}\cong\cO_X((n-1)H)$ (see \cite[Section 3.2]{I--P}). 

If $X$ is a Fano threefold with $i_X=2$ with fundamental line bundle $\cO_X(h)$, then $(X,\cO_X(h))$ is Del Pezzo. Conversely, if $(X,\cO_X(H))$ is a Del Pezzo threefold, then either $X$ is a Fano threefold with $i_X=2$ and $\cO_X(H)$ is its fundamental line bundle or $X\cong\p3$ and $\cO_X(H)\cong\cO_{\p3}(2)$ (see \cite[Section 3.3]{I--P}).

On the one hand, a vector bundle $\cE$ on $X\cong\p3$ is $\mu$--(semi)stable with respect to $\cO_{\p3}(1)$ if and only if it is $\mu$--(semi)stable with respect to $\cO_X(H)\cong\cO_{\p3}(2)$. Thus, the only strictly $\mu$--semistable rank two bundle $\cE$ with $c_1(\cE)=0$ on $X\cong \p3$ such that $h^1(\cE(-H))=h^1(\cE(-2))=0$ is $\cO_X^{\oplus2}$. On the other hand, there are no $H$--lines on $X$. We deduce that Theorem \ref{tFerrand} holds also in the case $X\cong\p3$ and $\cO_X(H)\cong\cO_{\p3}(2)$.
\end{remark}

\begin{remark}
The hypothesis $h^0(\cE) \ne 0$ is crucial in the proof of Theorem \ref{tFerrand} (and in Lemma \ref{lFerrand}). Indeed without this restriction, the description of strictly $\mu$--semistable rank two bundles $\cE$ with $c_1(\cE)=0$ and $h^1(\cE(-q^0_Xh))=0$ on a Fano variety $X$ with $\varrho_X\ge2$ is certainly much richer than the one in Theorem \ref{tFerrand}. 

As an example, consider the general hyperplane section $X \subset \p7$ of the Segre product $\p2\times\p2$ in $\p8$. The projections from $\p2\times\p2$ onto the two factors induce maps $p_i\colon X\to \p2$. It is well known that $\Pic(X)$ is generated by $\cO_X(h_i):=p_i^*\cO_{\p2}(1)$ for $i=1,2$ and $\cO_X(h)\cong\cO_X(h_1+h_2)$. For each pair of positive integers $a_1,a_2$ such that $\vert a_1-a_2\vert \le 1$, the extension
$$
0\longrightarrow\cO_X(a_1h_1-a_2h_2)\longrightarrow\cE\longrightarrow\cO_X(-a_1h_1+a_2h_2)\longrightarrow0
$$
gives a strictly $\mu$--semistable rank two bundle $\cE$ such that $c_1(\cE)=0$ and, thanks to \cite[Proposition 2.5]{C--F--M3}, $h^0(\cE)=h^1(\cE(-h))=0$; see also \cite[Proposition 3.5]{M--M--PL} for further details about this example.
\end{remark}

\begin{remark}
According to the definition in the introduction, the bundles $\cE$ considered in Theorem \ref{tFerrand} are not instanton bundles, because when $\varepsilon = 0$, the definition requires that $h^0(\cE)=0$. On the other hand, if we take the more general (and recent) definition of instanton given in \cite[Definition 1.3]{An--Cs}, then $\cE$ is instanton only when $\delta=1$ and $i_X=3$.
\end{remark}

\section{Examples}
\label{sExample}
In the examples below, we prove the existence of a rank two bundle $\cE$ satisfying the hypotheses of Lemma \ref{lFerrand} (hence also of Theorem \ref{tFerrand}) for each admissible pair $(e, \varepsilon)$, that is $e \in \{i_X, i_X-1\}$ and $0 \le  \varepsilon \le e-1$, and such that

\begin{equation}
\label{NonVanishing}
h^1(\cE(-(e-1)h)) \ne 0.
\end{equation}
When $\Pic(X)$ is generated by the fundamental line bundle $\cO_X(h)$, it follows by Lemma \ref{lHoppe} and \eqref{seqFerrand}, that such bundles are strictly $\mu$--semistable when $\varepsilon=0$ and $\mu$--stable when $\varepsilon=1$. There are also $\mu$--semistable (some strictly, some stable) examples for higher Picard number, see Example \ref{high} below. 

To this purpose, we will use Theorem \ref{tSerre}, starting from curves as in Theorem \ref{tFerrand}. In particular, in order to show the sharpness of the assertions (1) and (2) of the aforementioned theorem, we must construct examples starting from curves whose reduced structure is necessarily a disjoint union of $h$--lines. The existence of such subschemes is obvious when $i_X \ge 3$ and we refer the interested reader to \cite[Sections 3.4, 3.5, 4.2, 4.4 and 4.5]{I--P} for the case $i_X \le 2$. 

The starting curves will be disjoint unions of $h$--lines and/or $h$--conics in Examples \ref{eP3}, \ref{eQuartic} and \ref{eCubic}, while, in Example \ref{prime}, we will use a smooth rational curve on a BN--general K3 surface.  

\begin{example}[The cases $(e, \varepsilon) \in \{(i_X-1, i_X-2), (i_X, i_X-1)\}$] \hskip 3cm
\label{eP3}

This example is a generalization of the well--known construction in $\p3$ described in \cite[Examples 3.1.1 and 3.1.2]{Ha4}.

Let $Z_r \subset X$ be the union of $k\ge2$ pairwise disjoint smooth irreducible rational curves of $h$--degree $r$ where $\max\{1, 3-i_X\} \le r \le 2$. By construction we have $\omega_{Z_r}\cong\cO_{Z_r}\otimes\cO_{X}((r-3)h)$, hence 
$$\det(\cN_{{Z_r}/X}) \cong \omega_{Z_r} \otimes \cO_{X}(i_Xh) \cong \cO_{Z_r} \otimes \cO_{X}((r+i_X-3)h).$$
Thus Theorem \ref{tSerre} yields the existence of a rank two bundle $\cE$ with $c_1(\cE)=\varepsilon h$ fitting into \eqref{seqFerrand}, where $\varepsilon=r+i_X-3 \in \{i_X-2, i_X-1\}$.

Trivially $h^0(\cE)\ne0$. The cohomology of \eqref{seqFerrand} tensored by $\cO_X(-h)$ returns $h^0(\cE(-h))= h^0(\cI_{Z/X}((r+i_X-4)h))$. If either $i_X \le 2$ or $i_X=3$ and $r=1$, then the latter dimension is trivially $0$ because $Z_r\ne\emptyset$. If $i_X=3$ and $r=2$, then $r+i_X-4=1$ and $X\subset\p4$ is a smooth quadric hypersurface. It is easy to check that there are no hyperplanes in $\p4$ containing two disjoint conics in a smooth quadric. Similarly if $i_X=4$, then $r+i_X-4= r $ and $X \cong \p3$. Again there are no hypersurfaces of degree $r$ in $\p3$ containing two disjoint integral subschemes of degree $r$ when $r=1$ or $r=2$ and $k \ge 3$. When $(i_X, r, k) \ne (4, 2, 2)$ we conclude that $h^0(\cI_{Z/X}((r+i_X-4)h))=0$ again, hence $h^0(\cE(-h))=0$: in particular, if $\varrho_X=1$ and $\varepsilon=0$, then $\cE$ is $\mu$--semistable by Lemma \ref{lHoppe}. 

Let $e:=1+\varepsilon=r+i_X-2$. The cohomology of \eqref{seqFerrand} tensored by $\cO_X(-eh)$ returns $h^1(\cE(-eh))= h^1(\cI_{Z/X}(-h))$. The cohomology of \eqref{seqStandard} tensored by $\cO_X(-h)$ and the definition of $Z_r$ imply that the latter dimension is zero. The same argument also yields $h^1(\cE(-(e-1)h))=h^1(\cI_{Z/X})=k-1\ge1$.

Thus, when $(i_X, r, k) \ne (4, 2, 2)$, $\cE$ is a rank two bundle satisfying \eqref{NonVanishing} and the hypotheses of Lemma \ref{lFerrand} with $(e, \varepsilon) \in \{(i_X-1, i_X-2),(i_X, i_X-1)\}$ and $c_2(\cE)h=rk$.
\end{example}

\begin{example}[The cases $(e, \varepsilon) \in \{(i_X, i_X-2), (i_X, i_X-3), (i_X, i_X-4)\}$] \hskip 3cm
\label{eQuartic}
\indent If $i_X=1$, then there are no further cases besides the ones handled in Example \ref{eP3}. Thus we will assume $i_X \ge 2$ in what follows. In the fist part of this example, we assume that $X$ is not a Del Pezzo threefold of degree $1$, hence, as is well--known, either $X$ is not a Del Pezzo threefold of degree $2$ and $\cO_X(h)$ induces an aCM embedding $X \subset \p{N}$, or $\varphi_h : X \to \p3$ is a double cover ramified along a quartic surface.

Consider smooth pairwise disjoint $h$--conics $C_1,\dots, C_k\subset X$. 

We claim that there are smooth K3 surfaces $Y_i \in |-K_X|$ containing $C_i$, for all $1 \le i \le k$. The claim is clear in the case of the Del Pezzo threefold of degree $2$, since one can take as $C_i$ the inverse image under $\varphi_h$ of a general line in $\p3$ and as $Y_i$ the inverse image of a general quadric containing the line. In all the other cases we proceed as follows. It is easy to check that $\cI_{C_i/\p{N}}(i_X)$ is $0$--regular, hence the cohomology of the exact sequence
$$
0 \longrightarrow \cI_{X/\p {N}} \longrightarrow\cI_{C_i/\p {N}} \longrightarrow\cI_{C_i/ X} \longrightarrow 0,
$$
suitably tensored implies that the same holds for $\cI_{C_i/X}(i_Xh)$. Since $C_i$ is smooth and $i_X \ge 2$, it follows the existence of a smooth surface $Y_i\in\vert i_Xh\vert$ through $C_i$. 

Thus we know that $C_i$ is contained in a $K3$ surface in $Y_i\in\vert -K_X\vert$ without restrictions on $i_X \ge 2$. We set $\cO_{Y_i}(h_{Y_i}):=\cO_{Y_i}\otimes\cO_{X}(h)$, hence the degree of $Y_i$ is $h_{Y_i}^2=i_Xh^3$ and $h_{Y_i}C_i =2$. 

Recall that $\omega_{Y_i}\cong\cO_{Y_i}$, hence the adjunction formula on $Y_i$ implies $C_i^2=-2$.  Let $2 \le r \le i_X$ and consider the divisor ${Z_{i,r}}:=rC_i\subset Y_i$. The sequence \eqref{seqStandard} for the pair ${Z_{i,r}}\subset Y_i$ and the Riemann--Roch theorem on $Y_i$ imply that $\chi(\cO_{Z_{i,r}})=r^2$, whence $p_a({Z_{i,r}})=1-r^2\le -3$.

Since $(h_{Y_i}+C_i)C_i=0$ and $(h_{Y_i}+C_i)^2=h_{Y_i}^2+2>0$, it follows that $h_{Y_i}+C_i$ is big and nef, hence, for every $t \ge 1$ we have
\begin{equation}
\label{pert}
h^1(\cO_{Y_i}(t(h_{Y_i}+C_i)))=0 
\end{equation}
by the Kawamata--Viehweg vanishing theorem.

Moreover, the Riemann--Roch theorem on $Y_i$ implies
$$
h^0\big(\cO_{Y_i}(h+C_i)\big)\ge\frac{h_{Y_i}^2}2+3>\frac{h_{Y_i}^2}2+2=h^0\big(\cO_{Y_i}(h)\big),
$$
hence $C_i$ is not in the fixed locus of $\vert h_{Y_i}+C_i\vert$ and the general element in $\vert h_{Y_i}+C_i\vert$ does not intersect ${Z_{i,r}}$. Thus the same is true for $\vert r(h_{Y_i}+C_i)\vert$ and the adjunction formula on $Y_i$ gives 
$$
\omega_{Z_{i,r}} \cong \det(\cN_{Z_{i,r}/Y_i}) \cong \cO_{Z_{i,r}} \otimes \cO_{Y_i}({Z_{i,r}}) \cong \cO_{Z_{i,r}} \otimes \cO_{Y_i}(-rh_{Y_i}) \cong \cO_{Z_{i,r}} \otimes \cO_{X}(-rh).
$$
If $Z_r:=\bigcup_{i=1}^kZ_{i,r}$, then the adjunction formula on $X$ returns
$$
\det(\cN_{{Z_r}/X})\cong\cO_{Z_r}\otimes\cO_{X}((i_X-r)h)
$$
because the subschemes $Z_{i,r}$ are pairwise disjoint. Theorem \ref{tSerre} yields the existence of a rank two bundle $\cE$ with $c_1(\cE)=\varepsilon h$  fitting into \eqref{seqFerrand} where $\varepsilon=i_X-r$. 

As usual $h^0(\cE)\ne0$. We have $h^0(\cE(-h))=h^0(\cI_{{Z_r}/X}((i_X-r-1)h))$ and we claim that the latter dimension is zero. Indeed, if $i_X=4$ and $r=2$, then $i_X-r-1=1$. Since $p_a({Z_{i,r}})\le -3$, it follows that ${Z_{i,r}}$ is not contained in any plane inside $X\cong\p3$, hence the same is true for $Z_r$. If either $i_X \ne 4$ or $r\ne2$, then $i_X-r-1 \le 0$. We conclude also in this case that $h^0(\cI_{{Z_r}/X}((i_X-r-1)h))=0$ because $Z_r\ne\emptyset$. Thus the claim is proved, hence $h^0(\cE(-h))=0$: in particular, if $\varrho_X=1$ and $\varepsilon=0$, then $\cE$ is $\mu$--semistable by Lemma \ref{lHoppe}. 

We claim that $h^1(\cE(-i_Xh))=0$ and $h^1(\cE((1-i_X)h)) \ne 0$. To prove such a claim, we notice that the cohomology of 
\eqref{seqFerrand} tensored by $\cO_{X}(-th)$ yields
$$
h^1(\cE(-th))=h^1(\cI_{{Z_r}/X}((i_X-r-t)h))
$$
where $t\in\bZ$. If $t \ge i_X-1$, then $i_X-r-t \le -1$, hence the cohomologies of the sequence \eqref{seqStandard} for $X$ tensored by $\cO_{X}((i_X-r-t)h)$ and for $Y_i$ tensored by $\cO_{Y_i}((i_X-r-t)h_{Y_i})$ imply
\begin{align*}
h^1(\cI_{{Z_r}/X}((i_X-r-t)h))&=h^0(\cO_{Z_r}\otimes\cO_{X}((i_X-r-t)h))=\\
&=\sum_{i=1}^kh^0(\cO_{Z_{i,r}}\otimes\cO_{X}((i_X-r-t)h))=\\
&=\sum_{i=1}^kh^1(\cO_{Y_i}((i_X-r-t)h_{Y_i}-rC_i)).
\end{align*}

Taking $t=i_X$, the vanishing $h^1(\cO_{Y_i}(r(h_{Y_i}+C_i)))=0$ in \eqref{pert} and the Serre duality theorem on $Y_i$, yield $h^1(\cE(-i_Xh))=0$.

Now take $t=i_X-1$. On the one hand, the Riemann--Roch theorem on $Y_i$ yields
$$
\chi(\cO_{Y_i}((1-r)h_{Y_i}-rC_i))=\frac{(r-1)^2(h_{Y_i}^2+2)}2+1.
$$
On the other hand, $((r-1)h_{Y_i}+rC_i)C_i=-2$, hence $C_i$ is a fixed component of $\vert(r-1)h_{Y_i}+rC_i\vert$. Thanks to the Serre duality theorem, the Riemann--Roch theorem on $Y_i$ and \eqref{pert}, it follows that 
\begin{align*}
h^2(\cO_{Y_i}((1-r)h_{Y_i}-rC_i))&=h^0(\cO_{Y_i}((r-1)h_{Y_i}+rC_i))=\\
&=h^0(\cO_{Y_i}((r-1)(h_{Y_i}+C_i)))=\\
& = \chi(\cO_{Y_i}((r-1)(h_{Y_i}+C_i))=\frac{(r-1)^2(h_{Y_i}^2+2)}2+2.
\end{align*}
Thus
$$
h^1(\cO_{Y_i}((1-r)h_{Y_i}-rC_i)) = - \chi(\cO_{Y_i}((1-r)h_{Y_i}+rC_i)) + h^2(\cO_{Y_i}((1-r)h_{Y_i}-rC_i))=1,
$$
whence $h^1(\cE((1-i_X)h))=\sum_{i=1}^kh^1(\cO_{Y_i}((1-r)h_{Y_i}-rC_i)) = k$.

In particular $\cE$ is a rank two bundle satisfying \eqref{NonVanishing} and the hypotheses of Lemma \ref{lFerrand}  with $(e, \varepsilon) \in \{(i_X, i_X-4), (i_X, i_X-3), (i_X, i_X-2)\}$ and $c_2(\cE)h=2rk$.

Finally, consider $X$ a Del Pezzo $3$--fold of degree $1$. It follows by \cite[Proposition 3.2.4(i)]{I--P} that a general element $S \in |h|$ is a smooth Del Pezzo surface of degree $1$, that is the blow--up of the plane in $8$ points. Let $E_i, 1 \le i \le 8$ be the exceptional divisors, so that the $E_i$'s are pairwise disjoint $h$-lines on $X$.
If $Z=E_1 \cup \ldots \cup E_k$, for any $1 \le k \le 8$, we have that $\omega_Z \cong \cO_Z \otimes \cO_X(-2h)$ and therefore $\det(\cN_{Z/X}) \cong \cO_Z \otimes \cO_{X}$. Theorem \ref{tSerre} yields the existence of a rank two bundle $\cE$ with $c_1(\cE)=0$ fitting into \eqref{seqFerrand} and one can easily check that $H^1(\cE(-h))=H^1(\cE(-2h))=0$. Also, $\cE$ is $\mu$-semistable since, as is well known, $\Pic(X)$ is generated by $\cO_X(h)$.
\end{example}

\begin{example}[The cases $(e, \varepsilon) \in \{(i_X-1, i_X-3), (i_X-1, i_X-4)\}$] \hskip 3cm
\label{eCubic}

If $i_X \le 2$, then all the admissible cases are covered by Examples \ref{eP3} and \ref{eQuartic}. 
Thus $i_X \ge 3$ from now on, hence $\varrho_X=1$.

Consider smooth pairwise disjoint lines $L_1,\dots,L_k\subset X$. If $i_X=4$,  let $Y_{i}\subset \p3$ be a smooth cubic surface through $L_i$. If $i_X=3$, by combining the cohomologies of the twists of sequence
$$0  \longrightarrow \cO_{\p4}(-2) \longrightarrow \cI_{L_i/\p4} \longrightarrow \cI_{L_i/X} \longrightarrow 0$$
and of the Koszul complex resolving $\cI_{L_i/\p4}$, we deduce that $\cI_{L_i/X}(2h)$ is $0$--regular, hence there exists a smooth surface $Y_i\in\vert 2h\vert$ through $L_i$. 

By adjunction on $X$ we know that $Y_{i}$ is a Del Pezzo surface of degree $d=7-i_X$. In what follows we recall some facts about $Y_{i}$: we refer to \cite[Sections 24 and 25]{Man} for their proofs. The surface $Y_{i}$ is isomorphic to the blow up of $\p2$ at $9-d$ general points. As usual we denote by $\ell$ the pull--back of a general line in $\p2$ via the blowing up map and by $e_j$ the exceptional divisors for $1\le j\le 9-d$. Moreover we can always assume that $L_i=e_1$ for each $i$. If we set $\cO_{Y_{i}}(h_{Y_i}):=\cO_{Y_{i}}\otimes\cO_{X}(h)$, then $h_{Y_i}=3\ell-\sum_{j=1}^{9-d}e_j$. 

The subscheme ${Z_{i,r}}:=r L_i\subset Y_i\subset X$ with $2 \le r \le i_X-1$ is fixed and does not intersect the general element in $\vert3\ell-\sum_{j=2}^{9-d}e_j\vert$, hence
$$
\det(\cN_{Z_{i,r}/Y_i})\cong\cO_{Z_{i,r}}\otimes\cO_{Y_i}(rL_i)\cong\cO_{Z_{i,r}}\otimes\cO_{Y_i}(-rh_{Y_i}).
$$
Thus the adjunction formula on $Y_i$ yields
$$
\omega_{Z_{i,r}}\cong\cO_{Z_{i,r}}\otimes\cO_{Y_i}(-(r+1)h_{Y_i})\cong\cO_{Z_{i,r}}\otimes\cO_X(-(r+1)h).
$$
As in the previous example we set $Z_r:=\bigcup_{i=1}^kZ_{i,r}$ and notice that
$$
\det(\cN_{{Z_r}/X})\cong\cO_{Z_r}\otimes\cO_{X}((i_X-r-1)h)
$$
because the subschemes $Z_{i,r}$ are pairwise disjoint. Theorem \ref{tSerre} yields the existence of a rank two bundle $\cE$ with $c_1(\cE)=\varepsilon h$ fitting into \eqref{seqFerrand} and such that $\varepsilon=i_X-r-1 \in \{i_X-4, i_X-3\}$. 

Trivially $h^0(\cE)\ne0$. The cohomology of \eqref{seqFerrand} tensored by $\cO_{X}(-h)$ yields $h^0(\cE(-h))=h^0(\cI_{{Z_r}/X}((i_X-r-2)h))=0$, because $i_X-r-2\le 0$ and $Z_r\ne\emptyset$. Thus $h^0(\cE(-h))=0$, hence $\cE$ is $\mu$--semistable by Lemma \ref{lHoppe} if $\varepsilon=0$ because $\varrho_X=1$.

The cohomology of \eqref{seqFerrand} tensored by $\cO_{X}((1-i_X)h)$ yields 
\begin{equation}
\label{prima}
h^1(\cE((1-i_X)h))=h^1(\cI_{{Z_r}/X}(-rh)).
\end{equation}
The cohomology of the sequence \eqref{seqStandard} for $X$ tensored by $\cO_{X}(-rh)$ and for $Y_i$ tensored by $\cO_Y(-rh_{Y_i})$ imply
\begin{equation}
\label{seconda}
h^1(\cI_{{Z_r}/X}(-rh))=h^0(\cO_{Z_r}\otimes\cO_{X}(-rh))=\sum_{i=1}^kh^1(\cO_{Y_i}(-r(h_{Y_i}+L_i))).
\end{equation}
Since $(h_{Y_i}+L_i)L_i=0$ and $(h_{Y_i}+L_i)^2=1+i_X>0$, it follows that $h_{Y_i}+L_i$ is big and nef, hence $h^1(\cO_{Y_i}(-r(h_{Y_i}+L_i)))=0$ by the Kawamata--Viehweg vanishing theorem.
Therefore \eqref{prima} and \eqref{seconda} give 
$$
h^1(\cE((1-i_X)h))=\sum_{i=1}^kh^1(\cO_{Y_i}(-r(h_{Y_i}+L_i)))=0.
$$
We have $h^1(\cO_{Y_i}((r-2)h_{Y_i}+(r-1)L_i))=h^1(K_{Y_i}+(r-1)(h_{Y_i}+L_i))=0$ by the Kawamata--Viehweg vanishing theorem. Moreover the Serre duality theorem gives that $h^2(\cO_{Y_i}((r-2)h_{Y_i}+(r-1)L_i))=h^0(-(r-1)(h_{Y_i}+L_i))=0$. Hence
the exact sequence
$$0 \longrightarrow \cO_{Y_i}((r-2)h_{Y_i}+(r-1)L_i) \longrightarrow \cO_{Y_i}((r-2)h_{Y_i}+rL_i) \longrightarrow $$
$$\hskip 5cm \longrightarrow \cO_{L_i} \otimes \cO_{Y_i}((r-2)h_{Y_i}+rL_i) \longrightarrow 0$$
implies that $h^1(\cO_{Y_i}((1-r)h_{Y_i}-rL_i))=h^1(\cO_{L_i} \otimes \cO_{Y_i}((r-2)h_{Y_i}+rL_i))=h^1(\cO_{\p1}(-2))=1$.
Thus a similar computation as in \eqref{prima} and \eqref{seconda} yields
$$
h^1(\cE((2-i_X)h))=\sum_{I=1}^kh^1(\cO_{Y_i}((1-r)h_{Y_i}-rL_i))=k.
$$

In particular $\cE$ is a rank two bundle satisfying \eqref{NonVanishing} and the hypotheses of Lemma \ref{lFerrand}  with $(e, \varepsilon) \in \{(i_X-1, i_X-4), (i_X-1, i_X-3)\}$ and $c_2(\cE)h=rk$.  
\end{example}

We now give an example showing that rational curves of every even degree at least $4$ can occur in assertion (3) of Theorem \ref{tFerrand} when $\cO_X(h)$ is very ample.

\begin{example}[The case $(e, \varepsilon, i_X)=(1,0,1)$] \hskip 3cm
\label{prime}

If $X$ is a Fano threefold with $i_X=1$, then \eqref{RRgeneral} implies that $h^3$ is even: In fact, $\chi(\cO_X(h))=1+\frac{1}{2}h^3+\frac{1}{12}c_2(X)h$; on the other hand, $1=\chi(\cO_X)=-\frac{1}{24}K_Xc_2(X)=\frac{1}{24}c_2(X)h$, hence $c_2(X)h=24$ and $h^3$ is even. We call {\sl genus} of $X$ the number $g:=\frac{h^3}2+1$. The threefold $X$ is called {\sl prime} if $\Pic(X)$ is generated by $\cO_X(h)$ and {\sl of the principal series} if $\cO_X(h)$ is very ample. We have $3 \le g \le 12$, $g \ne 11$ for prime Fano threefolds of the principal series (see \cite{I}, \cite[Theorem 1.10]{Muk}, \cite[Theorems 6 and 7]{CLM}). 

A polarized K3 surface is a pair $(S,\mathcal L)$ where $S$ is a K3 surface and $\mathcal L$ is an ample line bundle on $S$. The polarized K3 surface $(S,\mathcal L)$ is called {\sl BN--general} if $h^0(\cA)h^0(\cB) < h^0(\mathcal L)$ for every non--trivial decomposition $\mathcal L\cong\cA\otimes\cB$. We will use below the fact, proved in \cite[(3.9) and Theorems 4.7, 5.5]{Muk}, that a BN--general polarized K3 surface $S \subset \p{g}$ can be realized as a hyperplane section of a prime Fano threefold $X \subset \p{g+1}$ of degree $2g-2$. 

Let $r \ge 1$ be an integer. We claim that for every $g$ as above, there are 
$$
C \subset S \subset X \subset \p{g+1}
$$ 
where $X$ is an anticanonically embedded prime Fano threefold of the principal series of genus $g$, $S$ is a smooth hyperplane section of $X$ (hence a K3 surface) and $C$ is a smooth irreducible rational curve of degree $2r$.

To this end observe that, by \cite[Theorem 1.1(iv)]{K1}, there is a smooth K3 surface $S \subset \p{g}$ containing a smooth irreducible rational curve $C$ of degree $2r$ and such that $\Pic(S)$ is freely generated by $\cO_X(H)$ and $\cO_X(C)$. Now \cite[Lemma 3.7(i)]{K2} implies that $(S,H)$ is a BN--general polarized K3 surface, hence, using the fact mentioned above, we can realize $S$ as a hyperplane section of a prime Fano threefold $X \subset \p{g+1}$. This proves the claim.
 
Now let $Z_r$ be the divisor $rC$ on $S$. We first show that 
\begin{equation}
\label{annu}
H^1(\cO_{Z_r}(-C))=0
\end{equation}
for each $r \ge 1$. In fact, since $C$ is connected and $H^1(\cO_S)=0$, we have that $H^1(\cO_S(-C))=0$. Now, from the exact sequence
\begin{equation}
\label{eq}
0 \longrightarrow \cO_S(-(r+1)C) \longrightarrow \cO_S(-C) \longrightarrow \cO_{Z_r}(-C) \longrightarrow 0
\end{equation}
the map
$$
H^2(\cO_S(-(r+1)C)) \longrightarrow H^2(\cO_S(-C))
$$ 
is dual to 
$$
H^0(\cO_S(C)) \longrightarrow H^0(\cO_S((r+1)C))
$$
and this is an isomorphism because $C$ is a base component of $\vert(r+1)C\vert$, since $C^2=-2$. Therefore the cohomology of \eqref{eq} gives \eqref{annu}. Next we claim that, for $r \ge 2$, we have
\begin{equation}
\label{banale}
\cO_{Z_r}\otimes \cO_S(h_S+rC) \cong \cO_{Z_r},
\end{equation}
where $\cO_S(h_S):=\cO_S\otimes\cO_X(h)$. To see this, first observe that we have an exact commutative diagram
$$
\xymatrix{& 0 \ar[d] & 0 \ar[d] & & \\ 0 \ar[r] & \cO_S(-rC) \ar[r] \ar[d] & \cO_S(-C) \ar[r] \ar[d] & \cO_{Z_{r-1}}(-C) \ar[r] & 0 &  \\ 0 \ar[r] & \cO_S \ar[r]^{\rm id}  \ar[d] & \cO_S \ar[r]  \ar[d] & 0 & \\ & \cO_{Z_r} \ar[r] \ar[d] & \cO_C \ar[r] \ar[d] & 0 & \\ & 0 & 0 & & }
$$
so that the snake lemma gives the exact sequence 
$$
0 \longrightarrow \cO_{Z_{r-1}}(-C) \longrightarrow \cO_{Z_r} \longrightarrow \cO_C \longrightarrow 0.
$$
Now, exactly as in \cite[Proof of Proposition 4.1]{BE}, the above gives rise to an exact sequence 
$$
0 \longrightarrow \cO_{Z_{r-1}}(-C) \longrightarrow \cO_{Z_r}^* \longrightarrow \cO_C^* \longrightarrow 1
$$
and applying \eqref{annu} we get that the restriction of line bundles gives an isomorphism
$$
\Pic(Z_r) = H^1(\cO_{Z_r}^*) \to H^1(\cO_C^*) = \Pic(C) \cong \bZ.
$$ 
Thus $\cO_{Z_r}\otimes \cO_S(h_{S}+rC) \cong \cO_{Z_r}$, because $(h_{S}+rC)C=0$, hence \eqref{banale} is proved.

Now the adjunction formula on $S$ and \eqref{banale} give
$$\omega_{Z_r} \cong \cO_{Z_r} \otimes \cO_S(rC) \cong \cO_{Z_r} \otimes \cO_S(-h_S) \cong \cO_{Z_r} \otimes \cO_X(-h).$$
Since $Z_r$ is locally complete intersection inside $X$, we deduce that 
$$\det(\cN_{Z_r/X}) \cong\cO_{Z_r}.$$ 
Also $H^1(\cO_X)=0$, hence Theorem \ref{tSerre} gives rise to a rank $2$ vector bundle $\cE$ on $X$ with $c_1(\cE)=0$, having a section $s$ with $Z_r = (s)_0$ and sitting in an exact sequence
\begin{equation}
\label{ult1}
0 \longrightarrow \cO_X \longrightarrow \cE \longrightarrow \cI_{Z_r/X} \longrightarrow 0.
\end{equation}
We have $H^i(\cO_X(-h))=0$ for $i = 0, 1$ and $H^0(\cI_{Z_r/X}(-h))=0$. Thus $h^0(\cE(-h))=0$, hence $\cE$ is strictly $\mu$--semistable by Lemma \ref{lHoppe}. Finally, to check $H^1(\cE(-h))=0$, using \eqref{ult1}, it remains to show that $H^1(\cI_{Z_r/X}(-h))=0$. To this end consider the exact sequence
\begin{equation}
\label{ult}
0 \longrightarrow \cO_X(-2h) \longrightarrow  \cI_{Z_r/X}(-h) \longrightarrow \cI_{Z_r/S}(-h) \longrightarrow 0.
\end{equation}
We have $(h_{S}+rC)C=0$ and $(h_{S}+rC)^2=h^3+2r^2>0$, hence $h_{S}+rC$ is big and nef and then
$$
H^1(\cI_{Z_r/S}(-h))=H^1(\cO_S(-h-rC))=0
$$
by the Kawamata--Viehweg vanishing theorem. Also $H^1(\cO_X(-2h))=0$, hence \eqref{ult} shows that 
$H^1(\cI_{Z_r/X}(-h))=0$. 
\end{example}

\begin{example}[Higher Picard number] \hskip 3cm
\label{high}

When $\rho_X \ge 2$, we have that $1 \le i_X \le 2$ and the possible values of $(e, \varepsilon, i_X)$ in Lemma \ref{lFerrand} are $(1, 0, 2), (2, 0, 2), (2, 1, 2)$ and $(1, 0, 1)$. 

For the case $(e, \varepsilon, i_X)=(2, 1, 2)$, one can take $\cE=\cE'(h)$, where $\cE'$ is as in \cite[Theorem 1.5]{A--C--G2} and $X$ is any smooth Fano threefold of index $2$. Here $\cE$ is $\mu$--stable.

For the case  $(e, \varepsilon, i_X) \in \{(1, 0, 2), (2, 0, 2)\}$ one can take for example $X=\mathbb P(T_{\p2})$, see \cite[Remark 3.2(iv)]{M--M--PL}. Here $\cE$ is strictly $\mu$--semistable.

For the case $(e, \varepsilon, i_X)=(1, 0, 1)$, consider for example $X = \p1 \times \p2$, together with the projections $\pi_1 : X \to \p1, \pi_2 : X \to \p2$ and set $h_1=\pi_1^*\cO_{\p1}(1), h_2=\pi_2^*\cO_{\p2}(1)$. Then $-K_X=2h_1+3h_2=h$ and $\Pic(X)$ is generated by $h_1$ and $h_2$. Take $k$ pairwise disjoint $h$-conics $C_i= \p1 \times \{x_i\}, 1 \le i \le k$ and set $Z=C_1 \cup \ldots \cup C_k$. We have that $\omega_Z \cong \cO_Z \otimes \cO_X(-h)$ and therefore $\det(\cN_{Z/X}) \cong \cO_Z \otimes \cO_{X}$. Theorem \ref{tSerre} yields the existence of a rank two bundle $\cE$ with $c_1(\cE)=0$ fitting into \eqref{seqFerrand}. Then $H^1(\cE(-h))=0$ follows by twisting \eqref{seqFerrand} by $\cO_X(-h)$,
$$0 \longrightarrow \cO_{X}(-h) \longrightarrow \cE(-h) \longrightarrow \cI_{Z/X}(-h) \longrightarrow 0$$
since $H^1(\cO_{X}(-h))=0$ by Kodaira vanishing and $H^1(\cI_{Z/X}(-h))=0$ as one can see by the exact sequence obtained by twisting \eqref{seqStandard} by $\cO_X(-h)$ and using the vanishing $H^0(\cO_Z \otimes \cO_X(-h))=0$.

To see that $\cE$ is $\mu$-semistable (hence clearly strictly $\mu$--semistable), we use \cite[Lemma 2.3]{A--C--G2}. For this purpose, we need to show that $H^0(\cE(-D))=0$ for each divisor $D$ on $X$ with $Dh^2 > \mu(\cE)=0$. Now, twisting \eqref{seqFerrand} by $\cO_X(-D)$, we get the exact sequence
$$0 \longrightarrow \cO_{X}(-D) \longrightarrow \cE(-D) \longrightarrow \cI_{Z/X}(-D) \longrightarrow 0$$
and, since $H^0(\cI_{Z/X}(-D)) \subseteq H^0(\cO_{X}(-D))$, we see that it is enough to show that $H^0(\cO_{X}(-D))=0$ for such divisors $D$. Next, write $D=a_1h_1+a_2h_2$ so that $Dh^2 > 0$ if and only if $3a_1+7a_2 > 0$. It follows that either $a_1 > 0$ or $a_2 > 0$ and then $H^0(\cO_{X}(-D))=H^0(\cO_{X}(-a_1h_1-a_2h_2))=0$ by K\"unneth's formula.

Similarly, other examples can be constructed on many Fano $3$-folds with $\rho_X \ge 2$ and $i_X=1$. 
\end{example}

\section{The case of the quadric}
\label{sQuadric}
We now focus our attention on the case $i_X=3$. Thus, from now on $X$ will denote the smooth quadric in $\p4$ and $\cO_X(h)$ its fundamental divisor.

If $\cE$ is a semistable rank two bundle on $X$ with $c_1(\cE)=0$, then it is either $\mu$--stable or $\cE\cong\cO_X^{\oplus2}$ (see Remark \ref{rGieseker}), hence we will focus on strictly $\mu$--semistable bundles. 

It follows by Theorem \ref{tFerrand} that such an $\cE$ satisfies $c_2(\cE)h\ge2$ and we will deal only with the case $c_2(\cE)h=2$ in what follows. In particular we will show, in Proposition \ref{corr} that each such bundle is obtained via the construction in Example \ref{eCubic}. 

In order to prove it, we denote by $\mathcal R$ the open subscheme of the Hilbert scheme $\Hilb^{2t+3}(X)$ whose points represent locally Cohen-Macaulay curves $Z \subset X$ with Hilbert polynomial $2t+3$. 

\begin{lemma}
\label{lBundle}
Let $\cE$ be a strictly $\mu$--semistable rank two bundle on $X$ with $c_1(\cE)=0$ and $c_2(\cE)h=2$. Then $h^0(\cE(-h))=0$, $h^0(\cE)=1$ and there is a unique $Z\in\mathcal R$ such that each non--zero section in $H^0(\cE)$ vanishes exactly on $Z$.
\end{lemma}
\begin{proof}
The $\mu$--semistability of $\cE$ implies that $h^0(\cE(-h))=0$ thanks to Lemma \ref{lHoppe}. The equality $h^0(\cE)=1$ and the existence and uniqueness of $Z$ follow as in the first lines of the proof of Theorem \ref{tFerrand}. Moreover, $\chi(\cO_Z)=2-\chi(\cE)=3$ thanks to the sequences \eqref{seqBundle}, \eqref{seqStandard} for the inclusion $Z\subset X$ and \eqref{RRgeneral}, hence $p_a(Z)=-2$, that is $Z\in\mathcal R$.
\end{proof}

In the following proposition we collect some helpful results on the schemes in $\mathcal R$.

\begin{proposition}
\label{pR}
Let $Z\subseteq X$ be a closed subscheme. Then $Z\in \mathcal R$ if and only if there is a smooth Del Pezzo quartic surface $Y\subset X$ and a line $L\subset Y$ such that $Z$ is the divisor $2L$ inside $Y$. Moreover, $Z$ is locally complete intersection inside $X$, $\omega_Z\cong\cO_Z\otimes\cO_X(-3h)$, $h^0(\cI_{Z/X}(h))=0$ and 
\begin{equation}
\label{HR}
h^1(\cI_{Z/X}(th))=\left\lbrace\begin{array}{ll} 
1\quad&\text{if $t=-1$,}\\
2\quad&\text{if $t=0$,}\\
0\quad&\text{otherwise.}
\end{array}\right.
\end{equation} 
\end{proposition}
\begin{proof}
If $Y\subset X$ is a smooth Del Pezzo quartic and $Z=2L$ on $Y$ for some line $L\subset Y$, then $\deg(Z)=2$ and the adjunction formula on $Y$ yields $p_a(Z)=-2$, hence $Z\in\mathcal R$. 

Conversely, let $Z \in \mathcal R$. Thus $Zh=2$ and $p_a(Z)=-2$, hence $Z$ is a double structure on a line $L\subset X$ necessarily. Since $Z_{red}=L$, we have that $\cI_{L/Z}=\mathcal N_Z$ the nilradical of $Z$ and the double structure gives that $\cI_{L/Z}^2=\mathcal N_Z^2=0$ (alternatively, $\cI_{L/Z}^2=0$ follows from the fact that $\cI_{L/X}^2 \subseteq \cI_{Z/X}$, see for example \cite[(1.3)]{bf}). We can choose the homogeneous coordinates $x_0,\dots, x_4$ in $\p 4$ so that the homogeneous ideal of $L$ inside $\bC[x_0,\dots,x_4]$ is $I_L=(x_0,x_1,x_2)$. 

As pointed out in \cite[Theorem 2.4]{N--N--S1} the homogeneous ideal $I_Z \subseteq \bC[x_0,\dots,x_4]$ of $Z$  has the  following form: there are two non--negative integers $\beta_1, \beta_2$ and there is a $3 \times 2$ matrix $B:=\left(b_{i,j}\right)_{\genfrac{}{}{0pt}{}{1\le i\le 3}{1\le j\le 2}}$ with coefficients $b_{i,j} \in \bC[x_3,x_4]$ that are homogeneous polynomials of degree $\beta_j$ and whose $2 \times 2$ minors define a codimension $2$ subscheme in $\p4$, so that $I_Z= I_L^2+(F_1,F_2)$ where, for $j=1, 2$, $F_j$ is the product of $(x_0,x_1,x_2)$ times the $j^{\mathrm{th}}$ column of $B$, hence it is a form of degree $\beta_j+1$. Since $\cI_{L/Z}^2=0$,
it follows that $\cI_{L/Z}$ is a line bundle on $L\cong\p1$. Thus the sequence \eqref{seqStandard} for the inclusion $L \subset Z$ becomes, 
$$
0\longrightarrow\cO_{\p1}(c-1)\longrightarrow\cO_Z\longrightarrow\cO_L\longrightarrow0
$$
for some integer $c$. The equality $p_a(Z)=-2$ implies $\chi(\cO_{\p1}(c-1))=2$, whence $c=2$, hence $\beta_1+\beta_2=2$ thanks to \cite[Lemma 2.6]{N--N--S1}. 

If $\beta_1=0$, then the homogeneous part of degree $2$ of $I_Z$ would be generated by $F_1x_3$, $F_1x_4$, $x_0^2$, $x_0x_1$, $x_0x_2$, $x_1^2$, $x_1x_2$, $x_2^2$ where $F_1:=b_{1,1}x_0+b_{2,1}x_1+b_{3,1}x_2$. Hence, if $F$ is the equation of $X$, since $Z \subset X$, we can write, 
$$F= a_3F_1x_3+a_4F_1x_4+a_{00}x_0^2+a_{01}x_0x_1+a_{02}x_0x_2+a_{11}x_1^2+a_{12}x_1x_2+a_{22}x_2^2$$
for some $a_k, a_{ij} \in \bC$. Then, for $i \in \{3, 4\}$, we have that $\partial F/\partial x_i =a_iF_1$, hence all partial derivatives of $F$ vanish at the point of intersection of the four hyperplanes defined by $F_1, \partial F/\partial x_j, 0 \le j \le 2$. This means that $X$ is singular, a contradiction. A similar contradiction occurs if $\beta_2=0$.

We deduce that $\beta_1=\beta_2=1$, hence $I_Z$ does not contain linear forms. The cohomology of the exact sequence
$$0 \longrightarrow \cI_{X/\p4} \longrightarrow \cI_{Z/\p4} \longrightarrow \cI_{Z/X} \longrightarrow 0$$
tensored by $\cO_{\p4}(t)$ and the isomorphism $\cI_{X/\p4}\cong\cO_{\p4}(-2)$ yield $h^0(\cI_{Z/X}(h))=0$ and \eqref{HR} by \cite[Corollary 3.2]{N--N--S1}. 

Moreover, $F_1$ and $F_2$ are quadratic forms such that $(F_1,F_2)$ is the ideal of a scheme $Y$ of dimension $2$ which is smooth along $L$ (see \cite[Theorem 2.4]{N--N--S1}). In particular, $I_Z$ is generated by the quadrics through $Z$, hence the base locus of the linear system of the quadrics through $Z$ is exactly $L$. It follows that two general quadrics through $Z$ are transversal outside $Z$: since $F_1$ and $F_2$ are transversal along $L$, the same is true for each general pair of quadrics through $Z$, hence we can assume that $Y$ is everywhere smooth. By adjunction on $\p4$ it follows that $Y$ is a Del Pezzo surface. Since $Y$ is smooth, the tangent space at every point of $Z$ has dimension at most $2$. On the other hand, $X$ is smooth and contains $Z$, hence there must be another smooth quadric, containing $Z$ and transversal to $X$. Therefore we can assume that $Y\subset X$.
As in Example \ref{eCubic} we can assume that $Y$ is the blow up of $\p2$ at $5$ general points. Let $\ell$ be the pull--back of a general line in $\p2$ and let $e_j$ be the exceptional divisors for $1\le j\le 5$: we can assume that $Z\subset Y$ is the Cartier divisor $2e_1$, hence it is locally complete intersection inside $X$ because $Y$ is smooth. As pointed out in Example \ref{eCubic}, the adjunction formula on $Y$ yields $\omega_{Z}\cong\cO_{Z}\otimes\cO_X(-3h)$.
\end{proof}

\begin{remark} 
\label{Ferrand} 
It follows by the above proof that, if $Z \in \mathcal R$, then there is a line $L \subset X$ and a surjection
$$\cI_{L/X}/\cI_{L/X}^2 \to \cI_{L/X}/\cI_{Z/X} \cong \cO_L(1).$$
Therefore $Z$ is obtained by the well-known Ferrand's construction \cite{f, bf}. 

Moreover, we have an explicit description of the homogeneous ideal of $Z \subset \p4$:
$$I_Z=(x_0^2, x_0x_1, x_0x_2, x_1^2, x_1x_2, x_2^2, b_{1,1}x_0+b_{2,1}x_1+b_{3,1}x_2,  b_{1,2}x_0+b_{2,2}x_1+b_{3,2}x_2)$$
where $b_{i,j} \in \bC[x_3,x_4]$ are homogeneous polynomials of degree $1$.
\end{remark} 

We can now prove the announced result.

\begin{proposition}
\label{corr}
There is a bijective correspondence between the following sets:
\begin{enumerate}
\item The set of strictly $\mu$--semistable rank two bundles $\cE$  on $X$ with $c_1(\cE)=0$ and $c_2(\cE)h=2$, up to isomorphism. 
\item The set of divisors of type $2L$ such that $L \subset Y\subset X$ where $L$ is a line and $Y$ is a smooth Del Pezzo quartic surface.
\end{enumerate}
\end{proposition}
\begin{proof}
Let $L \subset X$ be a line and let $Y \subset X$ be a smooth Del Pezzo quartic surface containing $L$. Then, as proved in Example \ref{eCubic}, Theorem \ref{tSerre} gives a rank two bundle $\cE$ on $X$, unique  up to isomorphism, with $c_1(\cE)=0, c_2(\cE)h=2Lh=2$ {and $h^0(\cE(-h))=0$}. It follows that $\cE$ is strictly $\mu$--semistable by Lemma \ref{lHoppe}. Thus $\cE$ is as in (1).

Vice versa let $\cE$ be as in (1). It follows by Lemma \ref{lBundle} that $\cE$ satisfies $h^0(\cE(-h))=0$, $h^0(\cE)=1$ and each non--zero section in $H^0(\cE)$ vanishes exactly on the same $Z \in \mathcal R$. Hence Proposition \ref{pR} shows that there is a smooth Del Pezzo quartic surface $Y \subset X$ and a line $L \subset Y$ such that $Z$ is the divisor $2L$ inside $Y$.  
\end{proof}
\medbreak

By Proposition \ref{pR} we obtain the following corollaries.

\begin{corollary}
\label{cTable}
Let $\cE$ be a strictly $\mu$--semistable rank two bundle on $X$ with $c_1(\cE)=0$ and $c_2(\cE)h=2$. Then the cohomology table of $\cE$ is as follows in the range $-3\le j\le 2$:
\begin{table}[H]
\centering
\bgroup
\def\arraystretch{1.5}
\begin{tabular}{cccccccc}
$i=3$ &\multicolumn{1}{|c}{1} & \multicolumn{1}{c}{0} & \multicolumn{1}{c}{0} & \multicolumn{1}{c}{0} & \multicolumn{1}{c}{0} & \multicolumn{1}{c}{0}  \\ 
$i=2$ &\multicolumn{1}{|c}{2} &  \multicolumn{1}{c}{1} & \multicolumn{1}{c}{0} & \multicolumn{1}{c}{0} & \multicolumn{1}{c}{0} & \multicolumn{1}{c}{0} \\ 
$i=1$ &\multicolumn{1}{|c}{0} &  \multicolumn{1}{c}{0} & \multicolumn{1}{c}{1} & \multicolumn{1}{c}{2} & \multicolumn{1}{c}{0} & \multicolumn{1}{c}{0}\\ 
$i=0$ &\multicolumn{1}{|c}{0} & \multicolumn{1}{c}{0} & \multicolumn{1}{c}{0} & \multicolumn{1}{c}{1} & \multicolumn{1}{c}{5} & \multicolumn{1}{c}{21}   \\ 
\cline{1-7}
&\multicolumn{1}{|c}{$j=-3$} &  \multicolumn{1}{c}{$j=-2$} &\multicolumn{1}{c}{$j=-1$} & \multicolumn{1}{c}{$j=0$} & \multicolumn{1}{c}{$j=1$} & \multicolumn{1}{c}{$j=2$} 
\end{tabular}
\egroup
\caption{The values of $h^i(\cE(jh))$.}
\end{table}
\end{corollary}
\begin{proof}
We already know that $h^0(\cE)=1, h^0(\cE(-h))=0$ and $Z\in \mathcal R$ by Lemma \ref{lBundle}. Hence $h^0(\cE(jh))=0$ for $j\le -1$. Also we can apply Proposition \ref{pR} and then \eqref{HR} yields $h^1(\cE(jh))=0$ for $j\ne -1,0$. Then the rest of the values in the table follow by the Serre duality theorem, that gives $h^i(\cE(jh))=h^{3-i}(\cE(-(3+j)h))$, and \eqref{RRgeneral}.
\end{proof}

\begin{remark} 
\label{rTable} 
Let $\cE$ be a $\mu$--semistable rank two bundle on $X$ with $c_1(\cE)=0$ and $c_2(\cE)h=2$. 

If $\cE$ is strictly $\mu$--semistable its cohomology table is given in Corollary \ref{cTable}. On the other hand, if $\cE$ is $\mu$--stable, it follows by Lemma \ref{lHoppe} that $h^0(\cE(jh))=0$ for all $j \le -1$ and by \cite[Corollary 2.4]{Ot--Sz} that $h^1(\cE(jh))=0$ for all $j \le -2$. It follows by the Serre duality theorem that $h^2(\cE)=h^3(\cE)=0$.

Hence, in both cases, $h^0(\cE)=h^1(\cE)-1$ by \eqref{RRgeneral}. Therefore $\cE$ is $\mu$--stable (resp. strictly $\mu$--semistable) if and only if $h^0(\cE)=0$ (resp. $1$), if and only if $h^1(\cE)=1$ (resp. $2$).

Moreover, $\cE$ is $2$--regular, hence there is a smooth curve $C\subset X$ which is the zero--locus of a section of $H^0(\cE(2h))$. Hence, as in \eqref{seqSerre}, we have an exact sequence
$$0 \longrightarrow \cO_X \longrightarrow \cE(2h) \longrightarrow \cI_{C/X}(4h) \longrightarrow 0.$$
Notice that $Ch=c_2(\cE(2h))h=10$ and $\omega_C\cong\cO_C\otimes\cO_X(h)$. The cohomologies of the above sequence tensored by $\cO_X(-4h)$, of the sequence \eqref{seqStandard} for the inclusion $C\subset X$ return
$$h^0(\cO_C)=h^1(\cI_{C/X})+1=h^1(\cE(-2h))+1=1,$$
hence $C$ is connected. We deduce that $C$ is the linear projection, from an outside point, of a canonical curve of genus $6$ contained in $\p5$ onto a hyperplane.

We have $h^0(\cI_{C/X}(2h))=h^0(\cE)$, hence $C$ corresponds, via the Serre correspondence, to a strictly $\mu$--semistable bundle if and only if it is contained in a single quadric.
\end{remark}

\begin{corollary}
\label{cR}
Let $Z\in \mathcal R$. Then $h^0(\cN_{Z/X})=7$ and $h^1(\cN_{Z/X})=0$.
\end{corollary}
\begin{proof}
Since $Z$ is locally complete intersection, we have an exact sequence 
\begin{equation}
\label{seqNormal}
0\longrightarrow \cN_{Z/Y}\longrightarrow\cN_{Z/X}\longrightarrow\cO_Z\otimes\cN_{Y/X}\longrightarrow0.
\end{equation}
We have $ \cN_{Z/Y}\cong\cO_Z\otimes\cO_Y(2e_1)$, $\cN_{Y/X}\cong\cO_Y(2h)$. 

We have $h^0(\cO_Y(2e_1))=1$. Moreover, $h^2(\cO_Y(2e_1))=h^0(\cO_Y(-2e_1-h))=0$ by the Serre duality theorem, hence the Riemann--Roch theorem on $Y$ yields $h^1(\cO_Y(2e_1))=0$. The cohomology of the sequence \eqref{seqStandard} for the pair $Z \subset Y$ tensored by $\cO_Y(2e_1)$ then yields
\begin{equation}
\label{ZY}
h^i(\cN_{Z/Y})=0
\end{equation}
for each $i$. Now the Serre duality theorem and the cohomology of the sequence \eqref{seqStandard} for the pair $Z \subset \p4$ imply
\begin{equation}
\label{YX1}
\begin{aligned}
h^1(\cO_Z\otimes\cN_{Y/X})&=h^1(\cO_Z\otimes\cO_{X}(2h))=\\
&=h^0(\cO_Z\otimes\cO_{X}(-5h))=h^1(\cI_{Z/X}(-5h))=0
\end{aligned}
\end{equation}
thanks to Proposition \ref{pR}. The cohomology of the sequence \eqref{seqStandard} for the pair $Z \subset Y$ tensored by $\cO_X(2h)$ returns 
\begin{equation}
\label{YX0}
h^0(\cO_Z\otimes\cN_{Y/X})=\chi(\cO_Z\otimes\cO_X(2h))=7
\end{equation}
thanks to the Riemann--Roch theorem on $Y$. The assertions on $h^i(\cN_{Z/X})$ then follow from the cohomology of \eqref{seqNormal} and the equalities \eqref{ZY}, \eqref{YX1}, \eqref{YX0}.
\end{proof}

As an immediate by--product of the above corollary, we deduce that the scheme $\mathcal R$ is smooth of dimension $7$. Moreover, as pointed out in the proof of Proposition \ref{pR}, the scheme $\mathcal R$ is dominated by an open subset of the Hilbert scheme of lines $L\subset X$ (which is isomorphic to $\p3$) times the set of $3\times 2$ matrices with entries in $\bC[x_3,x_4]$. Thus $\mathcal R$ is unirational.

\end{document}